\documentclass[ejs,10,authoryear]{imsart}

\RequirePackage[OT1]{fontenc}
\RequirePackage{amsthm,amsmath}
\RequirePackage{natbib}
\RequirePackage[colorlinks,citecolor=blue,urlcolor=blue]{hyperref}

\startlocaldefs
\numberwithin{equation}{section}
\theoremstyle{plain}
\newtheorem{theorem}{Theorem}%[section]
\newtheorem{lemma}{Lemma}%[section]
\newtheorem{corollary}{Corollary}
\theoremstyle{definition}
\theoremstyle{remark}
\newtheorem{remark}{Remark}

\endlocaldefs
\usepackage[left=4cm,right=4cm]{geometry}
\usepackage{natbib,epsfig,graphicx,color,amsmath,subfigure,amssymb,enumerate,multirow}
\begin{document}
\setattribute{journal}{name}{}
\begin{frontmatter}
\title{Asymptotics of a  Clustering Criterion for Smooth Distributions}
%\title{Test for Jumps in Semimartingale Models Using Clustering Criteria}
\runtitle{Clustering Criterion}
%\thankstext{T1}{Footnote to the title with the ``thankstext'' command.}

\begin{aug}
\author{\fnms{Karthik} \snm{Bharath}}\thanks{Ohio State University},%\thanksref{m1}},%$\ead[label=e1,mark]{karthik.bharath@uconn.edu}},
\author{Vladimir Pozdnyakov }\footnote{University of Connecticut}
\and
\author{Dipak. K. Dey}\footnotemark[\value{footnote}]

%\footnote[m1]{Ohio State University}
%\address[m2]{University of Connecticut}
%\printead{e1},
%\printead{e2}
%\and
%\printead{e3}\\
\end{aug}

\begin{abstract}
We develop a clustering framework for observations from a population with a smooth probability distribution function and derive its asymptotic properties. A clustering criterion based on a linear combination of order statistics is proposed. The asymptotic behavior of the point at which the observations are split into two clusters is examined. The results obtained can then be utilized to construct an interval estimate of the point which splits the data and develop tests for bimodality and presence of clusters.
%
%We develop a clustering framework, motivated by the problem of testing for jumps in continuous-time stochastic process models, and derive its asymptotic properties under a general setup. Our technique is applicable whenever we have data from a population with a smooth distribution function. We then propose an intuitive and easily verifiable clustering criterion, based on the Empirical Cross-over Function, which provides us with the requisite tools to develop a test for the presence of jumps. We illustrate the validity of our theory on the popular Merton and Kou models for asset pricing with the objective of investigating jumps occurring in these models as a phenomena which leads to the formation of clusters.
 \end{abstract}

%\keywords{Markov chains; extreme value theory}
%\ams{60F10}{60J10; 62F15; 62M05}


\end{frontmatter}
%%%%%%%%%%%%%%%%%%%%%%%%%%%%%%%%%%%%%% Model %%%%%%%%%%%%%%%%%%%%%%%%%%%%%%%%%%%%%%%%%%%%%%%%%%%%%%%%%%%%%%%5

\section{Introduction}
In this article, we develop a general framework for univariate clustering based on the ideas in \cite{JH2} for the case of observations from a population with smooth and invertible distribution function. Contrary to Hartigan's approach, which was based on a quadratic function of the observed data, our clustering criterion function possesses the advantage of being a linear combination of order statistics---in fact, it is a combination of trimmed sums and sample quantiles.

It is common in certain applications to assume that the data are taken from a population with smooth distribution function.  One important example is modeling in continuous-time mathematical finance, wherein observations are typically increments from a continuous-time stochastic process, and therefore, have smooth distributions because of presence of It\^{o} integral components. Keeping this in mind, we deviate from the Hartigan's framework and concentrate our attention on a function of the derivative of his {\it split function}. This approach permits us to obviate the existence of a  finite fourth moment assumption imposed by Hartigan in the asymptotic investigation of his criterion function---a second moment assumption at the cost of an additional smoothness condition on our criterion function suffices.  As an added benefit, this modification of Hartigan's approach provides us with the genuine possibility of extending our existing theory to more interesting scenarios involving dependent observations.

 The notion of a ``cluster'' has several reasonable mathematical definitions. As in \cite{JH2},  we adopt a definition based on determining a point which splits the data into clusters via maximizing the between cluster sums of squares. The main results in this article involve the asymptotic behavior of this particular point. The theoretical properties of $k$-means clustering procedure for the univariate and the multivariate cases have been extensively investigated. \cite{DP1}, \cite{DP2} proved strong consistency and asymptotic normality results in the univariate case. \cite{SGJ} proved some weak limit theorems under non-regular conditions for the univariate case. With the intention of having a more robust procedure for clustering, \cite{garcia1}, and \cite{garcia2} propose the trimmed $k$-means clustering and provide a central limit theorem for the multivariate case. { Throughout the article we are primarily concerned with the case $k=2$ on the real line. For extension of the split point approach to the case $k>2$ we refer readers to the discussion in \cite{JH2}.}

On a more practical note, our results enable us to construct an interval estimate of the point at which the data splits; this naturally allows us to develop simple tests for bimodality and presence (or absence) of clusters. Hypothesis tests for the presence (or absence) of clusters in a dataset has attracted considerable interest over the years.  One of the earliest work in this area was by \cite{HE}; they developed a univariate method to test the null hypothesis of normally distributed cluster against the alternative of a two-component mixture of normals. \cite{Wolfe} extended the work of \cite{HE} to the multivariate normal setup using MLE techniques and applied his method to Fisher's Iris data. Motivated by applications in market segmentation, \cite{SJ} proposed a test for clusters based on examining the within-groups scatter matrix. A dataset generated from a large-scale survey of lifestyle statements was considered and the objective was to capture heterogeneity in the distribution of the responses to an appropriate questionnaire. Based on statistics concerning mean distances, minimum-within clusters sums and the resulting F-statistics, \cite{Bock} presented several significance tests for clusters. In fact, he generalized some of the results in \cite{JH2} to the multivariate setup. In similar spirit, we note that in our method, the point which splits the data is invariant to scaling and translation of the data. This permits us to examine the behavior of the point under the null hypothesis of ``no cluster" and thereby construct a suitable test.

In \cite{JH3}, a popular test for unimodality, referred to as Dip test, was proposed and was applied to a dataset pertaining to the quality of 63 statistics departments. Indeed, their test did not possess good power against the specific bimodal alternative. More recently, \cite{HV} proposed a parametric test for bimodality based on the likelihood principle by using two-component mixtures. Their method was applied to investigate the modal structure of the cross-sectional distribution of per-capita log GDP across EU regions. Using the Kolmogorov-Smirnov and the Anderson-Darling statistics, \cite{astro}, performed a test for bimodality in the distribution of neutron-star masses. They compared the empirical cumulative distribution function to the distribution functions of a unimodal normal and a bimodal two-component normal mixture. Our results enable us to construct, on identical lines as the test for clusters, a test for bimodality. The test statistic, again, is based on the point at which the data is split---the split point is the same for all unimodal distributions with finite second moment and can hence be used as the test statistic.

In section \ref{clustering} we introduce the relevant constructs of our clustering framework: a theoretical criterion function and its zero followed by the empirical criterion function and its ``zero''. These quantities are of chief interest in this article. In section \ref{limit}, we prove limit theorems for the  empirical zero by examining the asymptotic behavior of the empirical criterion function and offer numerical verification of the limit results via simulation. Furthermore, we demonstrate the utility of our results on the popular {\tt faithful} dataset pertaining to eruption times for the Old Faithful geyser in Yellowstone National Park, Wyoming, USA. Finally, in section \ref{discussion} we highlight the salient features of our approach, note its shortcomings and comment on possible remedies and extensions.

%\newpage
%%%%%%%%%%%%%%%%%%%%%%%%%%%%%%%%%%%%%%%%%%%%%%%%%%%%%%%%%%%%%%%%%%%%%%%%%%%%%%%%%%%%%%%%%%%%%%%%
%%%%%%%%%%%%%%%%%%%%%%%%%%%%%%%%%%%%%%%%%%%%%%%%%%%%%%%%%%%%%%%%%%%%%%%%%%%%%%%%%%%%%%%%%%%%%%%%
%%%%%%%%%%%%%%%%%%%%%%%%%%%%%%%%%%%%%%%%%%%%%%%%%%%%%%%%%%%%%%%%%%%%%%%%%%%%%%%%%%%%%%%%%%%%%%%
%%%%%%%%%%%%%%%%%%%%%%%%%%%%%%%%%%%%%%%%%%%%%%%%%%%%%%%%%%%%%%%%%%%%%%%%%%%%%%%%%%%%%%%%%%%%%%%
\section{Clustering Criterion}\label{clustering}
In this section, based on Hartigan's approach, we propose an alternative clustering criterion and examine its properties. We first state our assumptions for the rest of the paper.
 \subsection{Assumptions}
  Let $W_1,W_2,\cdots,W_n$  be i.i.d. random variables with cumulative distribution function $F$. We denote by $Q$ the quantile function associated with  $F$. We make the following assumptions:
\begin{description}
\item [$A1$.] $F$ is invertible for $0<p<1$ and absolutely continuous with respect to Lebesgue measure with density $f$.
\item [$A2$.] $E(W_1)=0$ and $E(W_1^2)=1$.
%\item [$A3$.] For $0<a<b<1$, $F$ is continuously differentiable in $[a,b]$.
\item [$A3$.] $Q$ is twice continuously differentiable at any $0<p<1$.
\end{description}
Note that owing to assumption $A1$, the quantile function $Q$ is the regular inverse of $F$ and not the generalized inverse.
\subsection{Empirical Cross-over Function and Empirical Split Point}
Let us first consider the \emph{split function} that was introduced in \cite{JH2} for partitioning a sample into two groups.
The \emph{split function} of $Q$ at $p\in (0,1)$ is defined as
\begin{equation}\label{split*function}
 B(Q,p) = p(Q_l(p))^{2}+(1-p)(Q_u(p))^{2}-\left(\int_0^1Q(q)dq\right)^2,
\end{equation}
where
\begin{equation*}
Q_l(p)=\frac{1}{p}\int_{q< p}Q(q)dq=\frac{1}{p}E[W_1{\mathbb I}_{W_1<Q(p)}]
\end{equation*}
and
\begin{equation*}
Q_u(p)=\frac{1}{1-p}\int_{q\geq p}Q(q)dq =\frac{1}{1-p}E[W_1{\mathbb I}_{W_1\geq Q(p)}]
\end{equation*}
represent the conditional expectations of the random variables $W_i$ up to and from $Q(p)$. Here $\mathbb{I}_A$ denotes the indicator function of a set $A$. In our case since $EW_1=0$ the last term in the definition of the split function is $0$. A value $p_0$ which maximizes the split function is called the \emph{split point}. It is seen that if $Q$ is the regular inverse, as in our case, $p_0$ satisfies the equation
\begin{equation}
\label{splitone}
(Q_u(p_0)-Q_l(p_0))[Q_u(p_0)+Q_l(p_0)-2Q(p_0)]=0,
\end{equation}
where the LHS is the derivative of $B(Q,p)$. Evidently, $(Q_u(p)-Q_l(p)) > 0$ for all $0<p<1$ and we hence, for our purposes, consider the \emph{cross-over function},
\begin{equation}
\label{split}
G(p)= Q_l(p)+Q_u(p)-2Q(p),
\end{equation}
for examining clustering properties. From a statistical perspective, we would like to work with the empirical version of (\ref{split}).
We deviate here from Hartigan's framework and consider the \emph{empirical cross-over function}(ECF), defined in \cite{KB} as
\begin{equation}\label{G_function}
 G_n(p)=\frac{1}{k}\sum_{j=1}^k W_{(j)}-W_{(k)} +\frac{1}{n-k}\sum_{j=k+1}^n W_{(j)}-W_{(k+1)},
\end{equation}
for $\frac{k-1}{n}\leq p <\frac{k}{n}$
{ and
\begin{equation}\label{G_function*at*1}
 G_n(p)=\frac{1}{n}\sum_{j=1}^n W_{(j)}-W_{(n)},
\end{equation}
for $\frac{n-1}{n}\leq p <1$,  where $1 \leq k \leq n-1$.}

The random quantity $G_n$, represents the empirical version of (\ref{split}) and determines the split point for the given data. $G_n$ is an L-statistic with irregular weights and hence not amenable for direct application of existing asymptotic results for L-statistics.
{ Observe that
\begin{align*}
 G_n\left(\frac{0}{n}\right) &= W_{(1)}-W_{(1)}+\frac{1}{n-1}\sum_{j=2}^n W_{(j)}-W_{(2)}\quad \geq 0,\\
  G_n\left(\frac{n-1}{n}\right)&= \frac{1}{n}\sum_{j=1}^{n} W_{(j)}-W_{(n)} \quad \leq 0.
\end{align*}
This simple observation captures the typical behavior of the empirical cross-over function. It starts positive and then at some point crosses the zero line.  The index $k$ at which this change occurs determines the datum $W_{(k)}$ at which the split occurs. In \cite{KB}, it is shown that $G_n(p)$ is a consistent estimator of $G(p)$ for each $0<p<1$ and also that
$\sqrt{n}(G_n(p)-G(p))$ is asymptotically normal.}

{ We now introduce the \emph{empirical split point in range $[a,b]$, $0<a<b<1$}, the empirical counterpart of the $p_0$ as
\begin{equation*}
 p_n=p_n(a,b):=\left\{
 \begin{array}{l}
 0, \mbox{ if } G_n\left(\frac{k-1}{n}\right) < 0 \mbox{ $\forall k$ such that  } na < k < nb+1;\\
 \\
 1, \mbox{ if } G_n\left(\frac{k-1}{n}\right) > 0 \mbox{ $\forall k$ such that  } na < k < nb+1;\\
 \\
 \frac{1}{n}\left[\max\{na < k < nb : G_n\left(\frac{k-1}{n}\right)G_n\left(\frac{k}{n}\right) \leq 0\}\right], \mbox{ otherwise.}
 \end{array}
 \right.
\end{equation*}
The quantity $p_n$ is our estimator of $p_0$, the true split point (when it is in the range). If $p_n$ is equal to 0 or 1, we declare that the split point is outside the range. The asymptotic behavior of $p_n$ can be used for the construction of test for the presence of clusters in the observations, or for the estimation of the true split point.}

\begin{remark}
Let us provide some intuition behind Hartigan's split function.  The $k$-means clustering method for the case $k=2$ requires us to  minimize (with respect to $k^*$) the following within group sum  of squares:
\begin{align*}
W^*&=\sum_{i=1}^{k^*}\left(W_{(i)}-\frac{1}{k^*}\sum_{i=1}^{k^*}W_{(i)}\right)^2+\sum_{i=k^*+1}^{n}\left(W_{(i)}-\frac{1}{n-k^*}\sum_{i=k^*+1}^{n}W_{(i)}\right)^2\\
   &=\sum_{i=1}^{n}W_{(i)}^2-\frac{1}{k^*}\left(\sum_{i=1}^{k^*}W_{(i)}\right)^2-\frac{1}{n-k^*}\left(\sum_{i=k^*+1}^{n}W_{(i)}\right)^2.
\end{align*}
That is, minimizing $W^*$ is equivalent to maximizing
$$\frac{1}{k^*}\left(\sum_{i=1}^{k^*}W_{(i)}\right)^2+\frac{1}{n-k^*}\left(\sum_{i=k^*+1}^{n}W_{(i)}\right)^2$$
or
$$\frac{k^*}{n}\left(\frac{1}{k^*}\sum_{i=1}^{k^*}W_{(i)}\right)^2+\frac{n-k^*}{n}\left(\frac{1}{n-k^*}\sum_{i=k^*+1}^{n}W_{(i)}\right)^2,$$
which is basically an empirical version of Hartigan's split function (\ref{split*function}) and $k^*/n$ will be another version of empirical split point.

In this paper we proceed in parallel with \cite{JH2} (his Theorem 1 and Theorem 2) and prove consistency and asymptotic normality of $p_n$ under a uniqueness assumption. The theoretical conditions that guarantee the uniqueness of the split point is an open question. It is easy to see that for a unimodal symmetric distribution with a finite second moment, the split point is $1/2$, and for all unimodal symmetric light-tailed distributions that we checked the split point was unique. However, \cite{JH2} gives an example of a unimodal symmetric heavy-tailed distribution for which every point in $(0,1)$ is a split point. For a bimodal distribution the split point $p_0$ is typically unique and $Q(p_0)$ lies between the cluster means.

The presented results can be employed for testing ``no-clusters'' hypothesis, testing bimodality, and estimation of the split point. The extension of this technique to the case of $k$ clusters is discussed in \cite{JH2}. In our case, instead of one cross-over function one needs to introduce $k-1$ functions; the split point in this case will be a ($k-1$)-dimensional vector. To find this split point we then need to solve a system of $k-1$ equations. For instance, for partition of data into three groups we need to introduce two cross-over functions
\begin{align*}
G_{1n}\left[\frac{k_1-1}{n},\frac{k_2-1}{n}\right]&=\frac{1}{k_1}\sum_{i=1}^{k_1}W_{(i)}-W_{(k_1)}+\frac{1}{k_2-k_1}\sum_{i=k_1+1}^{k_2}W_{(i)}-W_{(k_1+1)},\\
G_{2n}\left[\frac{k_1-1}{n},\frac{k_2-1}{n}\right]&=\frac{1}{k_2-k_1}\sum_{i=k_1+1}^{k_2}W_{(i)}-W_{(k_2)}+\frac{1}{n-k_2}\sum_{i=k_2+1}^{n}W_{(i)}-W_{(k_2+1)},
\end{align*}
and, respectively, one needs to solve (in an appropriate sense) the following system of equations:
\begin{align*}
G_{1n}\left[\frac{k_1-1}{n},\frac{k_2-1}{n}\right]&=0,\\
G_{2n}\left[\frac{k_1-1}{n},\frac{k_2-1}{n}\right]&=0.
\end{align*}
We do not address the general $k>2$ cluster situation in this paper.

Finally,  as stated in the Introduction, let us remark on the main technical difference between Hartigan's assumptions and ours. Since we deal here with the derivative of the split function, we need a {\it stronger} smoothness condition (the second derivative of $G$ instead of the first one). In return, we work with trimmed means and  a {\it weaker} moment condition suffices (the finite second moment instead of the fourth one as in \cite{JH2}).
\end{remark}
\begin{remark}
Notice that if for constants $\alpha>0$ and $\beta$ and $i=1,\dots,n$,
\begin{equation*}
Z_i=\alpha W_i+\beta  ,
\end{equation*}
and we define $G_n^z$ to be the ECF based on $Z_i$, then,
\begin{align*}
G_n^z\left(\frac{k-1}{n}\right)    &= \frac{1}{k}\displaystyle \sum_{j=1}^k Z_{(j)}-Z_{(k)} + \frac{1}{n-k}\displaystyle \sum_{j=k+1}^n Z_{(j)}-Z_{(k+1)}\\
        &= \alpha \left[\frac{1}{k}\displaystyle \sum_{j=1}^k W_{(j)}-W_{(k)} + \frac{1}{n-k}\displaystyle \sum_{j=k+1}^n W_{(j)}-W_{(k+1)}\right]\\
        &= \alpha G_n\left(\frac{k-1}{n}\right),
\end{align*}
and therefore, $G_n^z$ and $G_n$ cross-over $0$ at the same point. Thus assumption $A2$ is not restrictive since $p_n$ is invariant to scaling and translation of the data; as it should be, since in a clustering problem scale and location changes of the data should not affect the clustering mechanism. We can hence quite safely assume that we are dealing with random variables with mean $0$ and variance $1$.
\end{remark}
%%%%%%%%%%%%%%%%%%%%%%%%%%%%%%%%%%%%%%%%%%%%%%%%%%%%%%%%%%%%%%%%%%%%%%%%%%%%%%%%%%%%%%%%%%
%%%%%%%%%%%%%%%%%%%%%%%%%%%%%%%%%%%%%%%%%%%%%%%%%%%%%%%%%%%%%%%%%%%%%%%%%%%%%%%%%%%%%%%%%%
\section{Main Results}\label{limit}
Let us start with the functional limit theorem for $U_n(p)=\sqrt{n}(G_n(p)-G(p))$ proved in \cite{KB}.
\begin{theorem}\label{G_n}
 Define
\begin{align*}
\theta_p=&\phantom{+}\frac{1}{p}W_1\mathbb{I}_{W_1<Q(p)}-\frac{1}{p}Q(p)\mathbb{I}_{W_1<Q(p)}\\
         &+\frac{1}{1-p}W_1\mathbb{I}_{W_1\geq Q(p)}-\frac{1}{1-p}Q(p)\mathbb{I}_{W_1\geq Q(p)}\\
         &+\frac{2\mathbb{I}_{W_1<Q(p)}}{f(Q(p))}.
\end{align*}
Under assumptions A1-A3,
\begin{equation*}
U_n(p)\Rightarrow U(p),
\end{equation*}
in the Skorohod space $D[a,b]$, $0<a<b<1$ equipped with the $J_1$ topology, where $U(p)$ is a Gaussian process with mean $0$ and covariance function given by
\begin{equation}\label{cov*function}
C(p,q)=Cov(U(p),U(q))=Cov(\theta_p,\theta_q).
\end{equation}
\end{theorem}

The next lemma states that the Gaussian process $U(p)$ allows a continuous modification. This fact will be employed, for example, to justify the usage of the mapping theorem (for instance, \cite{bill} and \cite{pollard1984}).
\begin{lemma}\label{continuity*of*U}
Under assumptions A1-A3, the centered Gaussian process $U(p),a\leq p\leq b$ with covariance function in (\ref{cov*function}) is continuous.
\end{lemma}
\begin{proof} First note that on the interval $[a,b]$ the functions (of $p$) $1/p$, $1/(1-p)$, $Q(p)$, $1/f(Q(p))=Q'(p)$, $E[W_1{\mathbb I}_{W_1<Q(p)}]$ and $E[W_1^2{\mathbb I}_{W_1<Q(p)}]$  are continuously differentiable. Second, the functions $\max(p,q)$ and $\min(p,q)$ are (globally) Lipschitz continuous on $[a,b]\times[a,b]$. Therefore, the covariance function $C(p,q)$ is Lipschitz continuous on $[a,b]\times[a,b]$; i.e., there is a constant $K$ such that
for all $p$, $q$, $p'$ and $q'$ from $[a,b]$ $$|C(p,q)-C(p',q')|\leq K(|p-p'|+|q-q'|).$$ Therefore,
\begin{align*}
E[U(p)-U(q)]^2&= C(p,p)+C(q,q)-2C(p,q)\\
              &\leq |C(p,q)-C(p,p)|+|C(p,q)-C(p,p)|\\
              &\leq 2K|p-q|.
\end{align*}
By Theorem 1.4 from \cite{adler1990} we get that $U(p)$ is continuous.
\end{proof}
This immediately leads us to the following important consequence.
\begin{corollary}
\label{uniform_Gn}
Under assumptions $A1-A3$, as $n\rightarrow \infty$,
\[
\sup_{a\leq p\leq b}\left|G_n(p)-G(p)\right|\overset{P}\rightarrow 0.
\]
\end{corollary}
\begin{proof} Since the functional  $\sup_{p\in [a,b]}|x(p)|$ is continuous on $C[a,b]$ equipped with the uniform metric, and the process $U(p)$ is continuous (that is,  $U(p)\in C[a,b]$ with probability 1), by the mapping theorem (\cite{pollard1984}, p. 70) we have
$$\sup_{a\leq p\leq b}\sqrt{n}\left|G_n(p)-G(p)\right|\Rightarrow\sup_{a\leq p\leq b}|U(p)|.$$
Therefore,
$$\sup_{a\leq p\leq b}\left|G_n(p)-G(p)\right|=\frac{1}{\sqrt{n}}\sup_{a\leq p\leq b}\sqrt{n}\left|G_n(p)-G(p)\right|\overset{P}\rightarrow 0.$$
\end{proof}
The empirical cross-over function is a step-function. The next lemma tells us that the jump at any $p\in(0,1)$ is $o_p(1/\sqrt{n})$.
\begin{lemma}\label{Gn_jump}
Under assumptions $A1-A3$, for $0<p<1$ and $\frac{k-1}{n} \leq p < \frac{k}{n}$, as $n\rightarrow \infty$,
\[
\left|G_n\left(\frac{k-1}{n}\right)-G_n\left(\frac{k-2}{n}\right)\right|=o_p\left(\frac{1}{\sqrt{n}}\right).
\]
\end{lemma}
\begin{proof}
\begin{align*}
%\begin{split}
\Big|G_n\left(\frac{k-1}{n}\right)-G_n\left(\frac{k-2}{n}\right)\Big|=\Big|\frac{1}{k}\displaystyle \sum_{i=1}^{k} W_{(i)}-W_{(k)}+\frac{1}{n-k}\displaystyle \sum_{k+1}^nW_{(i)}-W_{(k+1)}\\
%\parenthnewln{-}
-\frac{1}{k-1}\displaystyle \sum_{i=1}^{k-1} W_{(i)}+W_{(k-1)}-\frac{1}{n-k+1}\displaystyle \sum_{i=k}^{n}W_{(i)}+W_{(k)} \Big|.
%\end{split}
\end{align*}
Re-arranging terms, the RHS can written as
\begin{align*}
\Big|-\displaystyle \sum_{i=1}^{k-1}\frac{W_{(i)}}{k(k-1)}+\displaystyle \sum_{i=k+1}^n\frac{W_{(i)}}{(n-k)(n-k+1)}
&+W_{(k)}\left(\frac{n+1}{k(n-k+1)}\right)\\
&+(W_{(k-1)}-W_{(k+1)}) \Big|.
\end{align*}
Observe that, by the law of large numbers for trimmed sums (see \cite{stigler}),
\begin{equation*}
\frac{1}{(k-1)}\Big|\displaystyle \sum_{i=1}^{k-1}W_{(i)}\Big|\overset{P}\rightarrow \eta,
\end{equation*}
where $\eta$ is a constant and hence, as $n\rightarrow \infty$,
\[
\frac{1}{k(k-1)}\Big|\displaystyle \sum_{i=1}^{k-1}W_{(i)}\Big|=O_p\left(\frac{1}{n}\right)
\]
and similarly,
\begin{equation*}
\frac{1}{(n-k)(n-k+1)}\left|\displaystyle \sum_{i=k+1}^{n}W_{(i)}\right|=O_p\left(\frac{1}{n}\right).
\end{equation*}
Moreover, since $W_{(k)}\overset{P}\rightarrow Q(p)$, for $\frac{k-1}{n}\leq p <\frac{k}{n}$, we have that
\begin{equation*}
 \frac{(n+1)}{k(n-k+1)}|W_{(k)}|=O_p\left(\frac{1}{n}\right).
\end{equation*}
Suppose $M_n=\sup_{1\leq k \leq n}(W_{(k)}-W_{(k-1)})$, then from \cite{LD}, we have that
\begin{equation*}
M_n=O_p\left(\frac{\log n}{n}\right),
\end{equation*}
and therefore
\begin{equation*}
(W_{(k-1)}-W_{(k+1)})=O_p\left(\frac{\log n}{n}\right).
\end{equation*}
It is hence the case that the RHS is $o_p\left(\frac{1}{\sqrt{n}}\right)$. This concludes the proof.
\end{proof}

Now, we are ready to prove consistency of $p_n$. As in \cite{JH2} (Theorem 1) we require a uniqueness condition.
\begin{theorem}
\label{th1}
Assume $A1-A3$ hold. Suppose that $G(p)=0$ has a unique solution, $p_0$. Then for any $0<a<p_0<b<1$
\[
p_n\overset{P}\rightarrow p_0,
\]
as $n\rightarrow \infty$.
\end{theorem}
\begin{proof}
Note that by Cauchy-Schwarz inequality we have
$$[EW_1{\mathbb I}_{W_1\geq Q(p)}]^2\leq E[W_1^2{\mathbb I}_{W_1\geq Q(p)}]P[W_1\geq Q(p)].$$
Since the second moment of $W_1$ is finite, we obtain that $B(Q,0+)=B(Q,1-)=0$. That is, a nonnegative continuously differentiable split function $B(Q,p)$ has a unique maximum at $p_0$, and, as a result, $G(p)$ does change sign at $p_0$. Choose $a,b$ such that $0<a<p_0<b<1$. Because $G$ is continuous on $[a,b]$ we get that $G(p)>0$ for $a\leq p<p_0$ and $G(p)<0$ for $b\geq p>p_0$. Moreover, for any $\delta>0$ there exists an $\epsilon>0$ and $0<\delta'<\delta$ such that
$$G(p)>\epsilon\mbox{ for }a\leq p<p_0-\delta',$$
and
$$G(p)<-\epsilon\mbox{ for }b\geq p>p_0+\delta'.$$
By Corollary~\ref{uniform_Gn}, as $n \rightarrow \infty$
\[
P\left(\sup_{a \leq p \leq b}|G_n(p)-G(p)|<\frac{\epsilon}{2}\right)\rightarrow 1,
\]
 and therefore,
\[
P\left(\inf_{a\leq p<p_0-\delta'} G_n(p)>\frac{\epsilon}{2}\mbox{ and }\sup_{b\geq p>p_0+\delta'} G_n(p)<-\frac{\epsilon}{2}\right) \rightarrow 1.
\]
Using the result from Lemma~\ref{Gn_jump}, we obtain $$P(p_n\in[p_0-\delta',p_0+\delta'])\rightarrow 1.$$
Note that since
$$P(p_n\in[p_0-\delta,p_0+\delta])\geq P(p_n\in[p_0-\delta',p_0+\delta']),$$
we finally have
$$P(p_n\in[p_0-\delta,p_0+\delta])\rightarrow 1.$$
\end{proof}

Now, under an additional assumption that $G'(p_0)<0$ (cf. with Theorem 2 from \cite{JH2}) we will establish asymptotic normality of $p_n$. This result will be proved in three steps. First, we will establish that $p_n$ is in the $O_p(1/\sqrt{n})$ neighborhood of $p_0$. Then we will show that in this neighborhood $G_n(p)$ can be adequately approximated by a line with slope $G'(p_0)$. Finally, an approach based on Bahadur's general method (see p. 95, \cite{serfling}) will be employed to get the CLT for $p_n$.
\begin{lemma}
\label{pn*neighborhood}
Assume $A1-A3$ hold. Suppose that $G(p)=0$ has a unique solution, $p_0$, and $G'(p_0)<0$. If $a,b$ are such that $0<a<p_0<b<1$, then for any $\delta>0$ there exist $N$ and $C>0$ such that for all $n\geq N$
\[
P\left(|p_n-p_0|\leq \frac{C}{\sqrt{n}}\right)>1-\delta.
\]
\end{lemma}
\begin{proof}
Fix arbitrary $\delta>0$. Using Theorem~\ref{G_n} and mapping theorem we have
$$\sup_{a\leq p\leq b}\sqrt{n}\left|G_n(p)-G(p)\right|\Rightarrow\sup_{a\leq p\leq b}|U(p)|.$$
Therefore, for any $\delta>0$ there exist $N'$ and $C'>0$ such that for all $n>N'$ we have
\begin{equation}\label{bounded*in*probability}
P\left(\sup_{a\leq p\leq b}\left|G_n(p)-G(p)\right|<\frac{C'}{\sqrt{n}}\right)>1-\delta.
\end{equation}
By the same argument as in Theorem~\ref{th1}, $p_0$  is a unique split point, $0<p_0<1$, $G(p)>0$ for $p<p_0$ and $G(p)<0$ for $p>p_0$. Assumption $G'(p_0)<0$ tells us that in a neighborhood of $p_0$ the function $G(p)$ behaves like a line. Taking this into account we get that there exist $N>N'$ and $C>0$ such that for all $n>N$
$$G(p)>\frac{2C'}{\sqrt{n}}\mbox{ for }a\leq p<p_0-\frac{C}{\sqrt{n}},$$
and
$$G(p)<-\frac{2C'}{\sqrt{n}}\mbox{ for }b\geq p>p_0+\frac{C}{\sqrt{n}}.$$
Then by (\ref{bounded*in*probability}) we find that for all $n>N$
$$P\left(\inf_{a\leq p<p_0-C/\sqrt{n}} G_n(p)>\frac{C'}{\sqrt{n}}\mbox{ and }\sup_{b\geq p>p_0+C/\sqrt{n}} G_n(p)<-\frac{C'}{\sqrt{n}}\right)>1-\delta.$$
Therefore,
\[
P\left(|p_n-p_0|\leq \frac{C}{\sqrt{n}}\right)>1-\delta.
\]
\end{proof}

\begin{lemma}\label{line}
Assume $A1-A3$ hold. Suppose that $G(p)=0$ has a unique solution, $p_0$, and $G'(p_0)<0$. Then for any $C>0$
\[
\sup_{p\in I_n}\sqrt{n}\left|G_n(p)-G_n(p_0)-G'(p_0)(p-p_0)\right|\overset{P}\rightarrow 0, \textrm{ as }n\rightarrow \infty ,
\]
where $I_n=[p_0-\frac{C}{\sqrt{n}}, p_0+\frac{C}{\sqrt{n}}]$, and
\begin{align}\label{gprime}
G^{'}(p_0)&=\frac{1}{p_0}\left[Q(p_0)-Q_l(p_0)\right]
-\frac{1}{1-p_0}\left[Q(p_0)-Q_u(p_0)\right]-2Q'(p_0).
\end{align}
\end{lemma}
\begin{proof} Since the second derivative of $G(p)$ is uniformly continuous  on $p_0-C/\sqrt{n}\leq p\leq p_0-C/\sqrt{n}$ we have
\begin{align*}
G(p)-G(p_0)&=(p-p_0)G'(p)+O((p-p_0)^2)\\
           &=(p-p_0)G'(p)+O(1/n)
\end{align*}
It is hence sufficient to show that
$$
\sup_{p\in I_n}\sqrt{n}\left|[G_n(p)-G(p)]-[G_n(p_0)-G(p_0)]\right|\overset{P}\rightarrow 0,
$$
or that for any $\epsilon>0$ and $\delta>0$ there exists $N$ such that for all $n>N$
$$
P\left(\sup_{p\in I_n}\sqrt{n}\left|[G_n(p)-G(p)]-[G_n(p_0)-G(p_0)]\right|>\epsilon\right)<\delta.
$$
Take arbitrary $\delta'>0$.
The functional  $\sup_{p\in [p_0-\delta',p_0+\delta']}|x(p)-x(p_0)|$ is continuous on $C[a,b]$ equipped with the uniform metric. Therefore, Theorem~\ref{G_n} and the mapping theorem informs us that
$$
\sup_{p_0-\delta'\leq p\leq p_0+\delta'}\sqrt{n}\left|[G_n(p)-G(p)]-[G_n(p_0)-G(p_0)]\right|\Rightarrow \sup_{p_0-\delta'\leq p\leq p_0+\delta'}\left|U(p)-U(p_0)\right|.
$$

Since for  all sufficiently large $n$ we have
\begin{align*}
\sup_{p\in I_n}\sqrt{n}\left|[G_n(p)-G(p)]-[G_n(p_0)-G(p_0)]\right|&\\
\leq \sup_{p_0-\delta'\leq p\leq p_0+\delta'}\sqrt{n}\left|[G_n(p)-G(p)]-[G_n(p_0)-G(p_0)]\right| \quad \text{a.s.},
\end{align*}
it is indeed the case that
\begin{align*}
P\Big(\sup_{p\in I_n}\sqrt{n}\left|[G_n(p)-G(p)]-[G_n(p_0)-G(p_0)]\right|>\epsilon\Big)&\\
\leq P\Big(\sup_{p_0-\delta'\leq p\leq p_0+\delta'}\sqrt{n}\left|[G_n(p)-G(p)]-[G_n(p_0)-G(p_0)]\right|>\epsilon\Big)&.
\end{align*}
As a consequence,
\begin{align*}
\limsup_n & P\Big(\sup_{p\in I_n}\sqrt{n}\left|[G_n(p)-G(p)]-[G_n(p_0)-G(p_0)]\right|>\epsilon\Big)\\
\leq&P\Big(\sup_{p_0-\delta'\leq p\leq p_0+\delta'}\left|U(p)-U(p_0)\right|>\epsilon\Big).
\end{align*}
Because the Gaussian process $U(p)$ is continuous, $\sup_{p_0-\delta'\leq p\leq p_0+\delta'}\left|U(p)-U(p_0)\right|\to 0$ with probability 1 as $\delta'\to 0$, and, therefore, it converges to 0 in probability. Choosing $\delta'$ small enough we can make
$$\limsup_n P\Big(\sup_{p\in I_n}\sqrt{n}\left|[G_n(p)-G(p)]-[G_n(p_0)-G(p_0)]\right|>\epsilon\Big)< \delta.$$
\end{proof}
\begin{lemma}\label{p_n}
Assume $A1-A3$ hold. Suppose that $G(p)=0$ has a unique solution, $p_0$, and $G'(p_0)<0$. If $a,b$ are such that $0<a<p_0<b<1$ then as $n\rightarrow \infty$
\[
	p_n=p_0-\frac{G_n(p_0)}{G'(p_0)}+o_p(n^{-1/2}),
\]
where $G'(p)$ is as defined in (\ref{gprime}).
\end{lemma}
\begin{proof}

Consider the line $G_n(p_0)+G'(p_0)(p-p_0)$. Let random variable $p^*$ be the solution of
\[
G_n(p_0)+G'(p_0)(p-p_0)=0,	
\]
that is,
\begin{equation}\label{pstar}
	p^*=p_0-\frac{G_n(p_0)}{G'(p_0)}.
\end{equation}
From Theorem \ref{G_n}, we know that $$G_n(p_0)-G(p_0)=G_n(p_0)=O_p(n^{-1/2})$$ and we hence have that $p^*=p_0+O_p(n^{-1/2})$. By Lemma~\ref{pn*neighborhood} $p_n$ is also in a $O_p(1/\sqrt{n})$ neighborhood of $p_0$.  By Lemma~\ref{line} in $O_p(1/\sqrt{n})$ neighborhood of $p_0$ uniformly $$G_n(p)=G_n(p_0)+G'(p_0)(p-p_0)+o_p(n^{-1/2}),$$ by which we can claim that $p_n=p^*+o_p(n^{-1/2})$.
The result follows by substituting for $p^*$ in (\ref{pstar}).
\end{proof}
This immediately give us the final result.
\begin{theorem}\label{normal}
Assume $A1-A3$ hold. Suppose that $G(p)=0$ has a unique solution, $p_0$, and $G'(p_0)<0$. If $a,b$ are such that $0<a<p_0<b<1$ then as $n\rightarrow \infty$,
\[
\sqrt{n}(p_n-p_0)\Rightarrow N\left(0,\frac{Var(\theta_{p_0})}{G'^2(p_0)}\right),
\]
where $\theta_{p_0}$ is as defined in Theorem \ref{G_n}.
\end{theorem}

\subsection{Numerical Verification}
In this section we provide verification of our results regarding the asymptotic normality of $p_n$ along the lines of Table 1 in \cite{JH2}. Since our split point $p_0$ coincides with Hartigan's split point (maximum of $B(Q,p)$), it is to be expected that our empirical split point $p_n$ behaves asymptotically similar to his. Hartigan verifies his results when observations are obtained from a $N(0,1)$ population---a population with smooth distribution function; we do the same and note that the asymptotic mean and the variance of $p_n$ agree with his.

It is a simple exercise to ascertain that, for the normal case, the split point $p_0$ is $0.5$, $G'^2(0.5)\approx 3.34$ and $Var(\theta_{0.5})=2 \pi -4\approx 2.283$. Consequently, we observe that the asymptotic variance of $\sqrt{n}(p_n-0.5)$ is approximately $0.69$. The table below corroborates our theoretical results.
\begin{table*}[th]
\begin{center}
\caption{Simulated mean and variance of $\sqrt{n}(p_n-0.5)$ for different sample sizes for the normal case. $1000$ simulations were performed for each sample size.}
\vspace{2mm}
%\resizebox{16cm}{1.5cm}{
\begin{tabular}{|c||c|c|c|c|}
\hline
%\textbf{Random variables}&\multicolumn{3}{|c|}{$\mathbf{N(0,1)}$}&\multicolumn{3}{|c|}{$\mathbf{Exp(1)}$}\\
%\multicolumn{1}{|c||}{\multirow{2}{*}{Results}}\\
%\hline
Sample sizes& $n=100$ & $n=300$& $n=500$& $n=1000$\\ \hline
%p-value (K-S Test) &$0.435$& $0.5369$&$0.516$ &$0.048$ & $0.421$ & $0.273$  \\ \hline
%p-value (Wilcox Test) &$0.98$& $0.355$&$0.899$  \\ \hline
Simulated Mean &$0.506$& $0.504$&$0.501$ &$0.502$ \\ \hline
Simulated Variance &$0.614$& $0.646$&$0.700$ &$0.691$ \\ \hline
\end{tabular}
\end{center}
\end{table*}

\subsection{An Example: Confidence Interval Estimation}
We demonstrate here how Theorem~\ref{normal} can be employed to construct approximate confidence intervals (CI) for a theoretical split point.
We consider a classical example of bimodal distribution---the variable ``eruption" in the data set {\tt faithful} available in {\tt R} package {MASS}.
The data set contains 272 measurements of the duration of eruption for the Old Faithful geyser in Yellowstone National Park, Wyoming, USA. 

First, we plot the ECF for the variable "eruption"; the plot is given in Figure~1.  
\begin{figure}[ht]
\includegraphics[scale=0.4,angle=90]{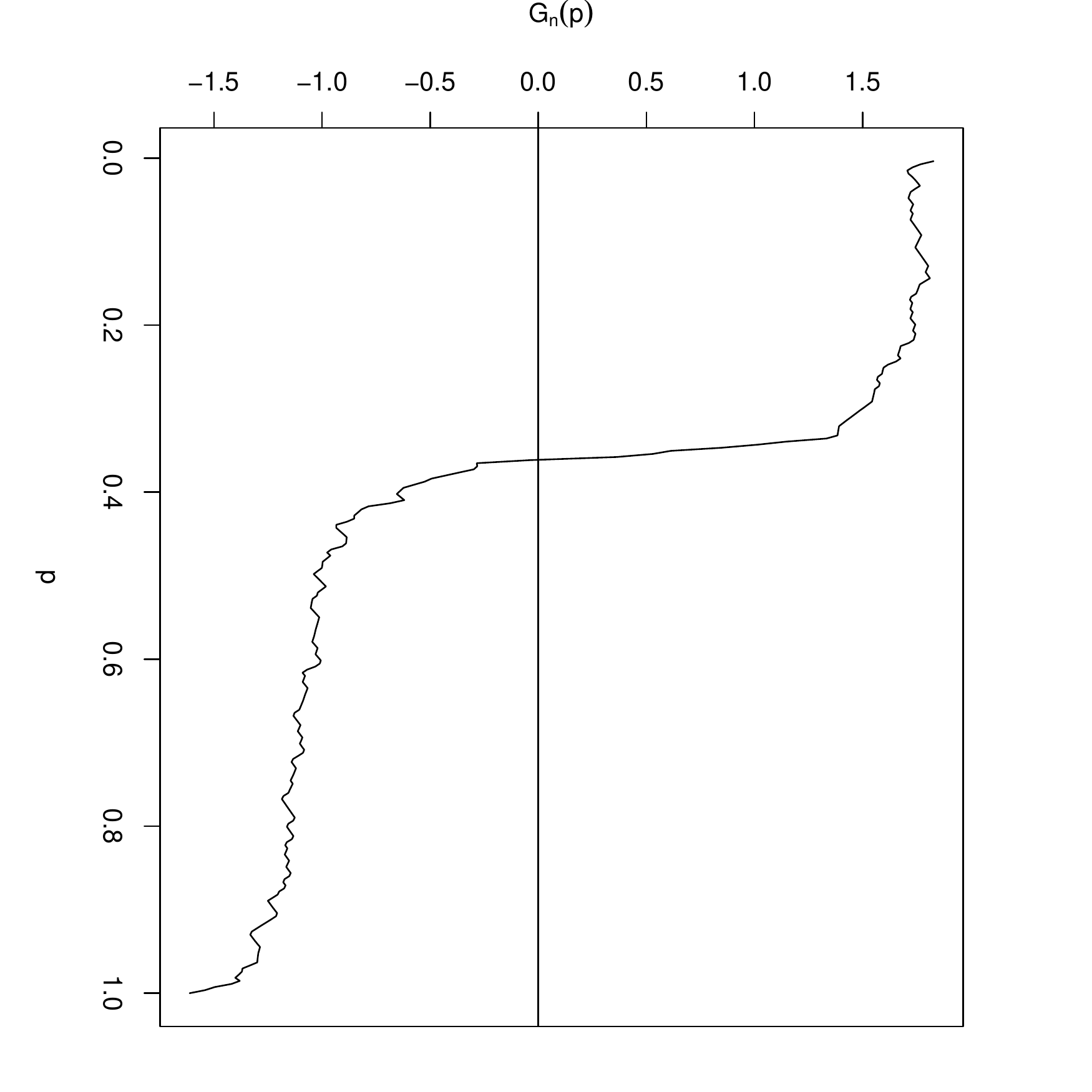}
\caption{Empirical Cross-over Function $G_n(p)$ for data set {\tt faithful}.}
\label{faithful}
\end{figure}
We can see that $G_n(\cdot)$ is generally a decreasing function
that crosses zero line once, far away from $0$ and $1$: the end-points of its domain which is the $(0,1)$ interval. Thus our point estimate of theoretical split point is $p_n=97/272\approx.357$. 

Now, to construct an approximate CI for $p_0$ we need to estimate $Var(\theta_{p_0})/G'^2(p_0)$. A straightforward (but rather tedious) calculation shows that this quantity explicitly depends on the following terms:
$p_0$, $Q(p_0)$, $f(Q(p_0))$, $Q_l(p_0)$, $Q_u(p_0)$,
\begin{equation*}
B_l(p_0)=\frac{1}{p_0}E[W_1^2{\mathbb I}_{W_1<Q(p_0)}], \mbox{ and } B_u(p_0)=\frac{1}{1-p_0}E[W_1^2{\mathbb I}_{W_1\geq Q(p_0)}].
\end{equation*}
We estimate these terms as follows:
$$p_0\approx p_n,\quad Q(p_0)\approx W_{(98)},$$
$$Q_l(p_0)\approx\frac{1}{98}\sum_{i=1}^{98}W_{(i)},\quad Q_u(p_0)\approx\frac{1}{272-98}\sum_{i=99}^{272}W_{(i)},$$
$$B_l(p_0)\approx\frac{1}{98}\sum_{i=1}^{98}W_{(i)}^2,\quad B_u(p_0)\approx\frac{1}{272-98}\sum_{i=99}^{272}W_{(i)}^2.$$
Finally, $f(Q(p_0))$ is estimated by $\hat{f}(W_{(98)})$, where $\hat{f}$ comes from the standard {\tt R} function {\tt density}. 
As a result, for instance, the 95\% confidence interval for a theoretical split point $p_0$ is given by
$$.357\pm .057.$$

%\newpage
%%%%%%%%%%%%%%%%%%%%%%%%%%%%%%%%%%%%%%%%%%%%%%%%%%%%%%%%%%%%%%%%%%%%%%%%%%%%%%%%%%%%%%%%%%%%%%%%%
%%%%%%%%%%%%%%%%%%%%%%%%%%%%%%%%%%%%%%%%%%%%%%%%%%%%%%%%%%%%%%%%%%%%%%%%%%%%%%%%%%%%%%%%%%%%%%%%%
\section{Discussion}\label{discussion}
Admittedly, the definition of the empirical split point ``in the range $[a,b]$'' might appear a bit artificial. But we still believe the results can be useful in practical applications. For instance, as with the {\tt faithful} data, if we know that the distribution at hand is bimodal, and we want to estimate a split point between two clusters, it is safe to assume that the split point is in the range between two modes.

{ It turns out that the behavior of the cross-over function when it is close to 0 or 1 can be rather complicated; under some natural assumptions, it can be shown that $\limsup_{p\uparrow 1} G(p)<0$. For example, it is true if $W_1$ is bounded from above. When $Q(1-)=+\infty$, the following condition
\begin{equation}\label{mises}
\limsup_{x\to\infty}\frac{E(W_1{\mathbb I}_{W_1\geq x})}{xP(W_1\geq x)}<2
\end{equation}
is sufficient for $\limsup_{p\uparrow 1} G(p)<0$. It is easy to see that, for instance, distributions with regularly varying tails and $EW_1^{2+\epsilon}<\infty$, for $\epsilon>0$, satisfy (\ref{mises}). However, it is possible to construct a distribution with a ``bumpy'' tail for which $\limsup_{p\uparrow 1} G(p)\geq 0$. Consequently, it is suggestive that any extension of definition of $p_n$ to the entire interval $(0,1)$ will require some additional assumptions.}

\section{Acknowledgements} We are  thankful to the Associate Editor and referee for careful reading of the manuscript, thoughtful criticisms and very helpful suggestions that allowed us to significantly improve the presentation of our results. We also want to thank Zhiyi Chi for discussions on the behavior of the cross-over function around 0 and 1.

\bibliography{ref}

\begin{thebibliography}{20}
% BibTex style file: imsart-nameyear.bst, 2010-01-14
% Default style options (sort=1,type=nameyear).
% Used options (sort=1,type=nameyear).

\bibitem[\protect\citeauthoryear{Adler}{1990}]{adler1990}
\begin{barticle}[author]
\bauthor{\bsnm{Adler},~\bfnm{R~J}\binits{R.~J.}}
(\byear{1990}).
\btitle{{An Introduction to Continuity, Extrema, and Related Topics for General
  Gaussian Processes}}.
\bjournal{{Lecture Notes-Monograph Series}}
\bvolume{12}.
\end{barticle}
\endbibitem

\bibitem[\protect\citeauthoryear{Arnold}{1979}]{SJ}
\begin{barticle}[author]
\bauthor{\bsnm{Arnold},~\bfnm{S~J}\binits{S.~J.}}
(\byear{1979}).
\btitle{{A Test for Clusters.}}
\bjournal{{Journal of Marketing Research}}
\bvolume{16}
\bpages{545-551}.
\end{barticle}
\endbibitem

\bibitem[\protect\citeauthoryear{Bharath, Pozdnyakov and Dey}{2012}]{KB}
\begin{barticle}[author]
\bauthor{\bsnm{Bharath},~\bfnm{K}\binits{K.}},
  \bauthor{\bsnm{Pozdnyakov},~\bfnm{V}\binits{V.}} \AND
  \bauthor{\bsnm{Dey},~\bfnm{D~K}\binits{D.~K.}}
(\byear{2012}).
\btitle{{Asymptotics of Empirical Cross-over Function}}.
\bjournal{{Unpublished manuscript. Arxiv:1112.3427v3}}.
\end{barticle}
\endbibitem

\bibitem[\protect\citeauthoryear{Billingsley}{1968}]{bill}
\begin{bbook}[author]
\bauthor{\bsnm{Billingsley},~\bfnm{P}\binits{P.}}
(\byear{1968}).
\btitle{{Convergence of Probability Measures}}.
\bpublisher{John Wiley and Sons, New York}.
\end{bbook}
\endbibitem

\bibitem[\protect\citeauthoryear{Bock}{1985}]{Bock}
\begin{barticle}[author]
\bauthor{\bsnm{Bock},~\bfnm{H~H}\binits{H.~H.}}
(\byear{1985}).
\btitle{{On Some Significance Tests in Cluster Analysis.}}
\bjournal{{Journal of Classification}}
\bvolume{2}
\bpages{77-108}.
\end{barticle}
\endbibitem

\bibitem[\protect\citeauthoryear{Cuesta-Albertos, Gordaliza and
  Matr\'{a}n}{1997}]{garcia2}
\begin{barticle}[author]
\bauthor{\bsnm{Cuesta-Albertos},~\bfnm{J~A}\binits{J.~A.}},
  \bauthor{\bsnm{Gordaliza},~\bfnm{A}\binits{A.}} \AND
  \bauthor{\bsnm{Matr\'{a}n},~\bfnm{C}\binits{C.}}
(\byear{1997}).
\btitle{{Trimmed $k$-means: An Attempt to Robustify Quantizers.}}
\bjournal{{Annals of Statistics}}
\bvolume{25}
\bpages{553-576}.
\end{barticle}
\endbibitem

\bibitem[\protect\citeauthoryear{Devroye}{1981}]{LD}
\begin{barticle}[author]
\bauthor{\bsnm{Devroye},~\bfnm{L}\binits{L.}}
(\byear{1981}).
\btitle{{Laws of Iterated Logarithm for Order Statistics of Uniform Spacings}}.
\bjournal{{Annals of Probability}}
\bvolume{9}
\bpages{860-867}.
\end{barticle}
\endbibitem

\bibitem[\protect\citeauthoryear{Engleman and Hartigan}{1969}]{HE}
\begin{barticle}[author]
\bauthor{\bsnm{Engleman},~\bfnm{L}\binits{L.}} \AND
  \bauthor{\bsnm{Hartigan},~\bfnm{J~A}\binits{J.~A.}}
(\byear{1969}).
\btitle{{Percentage Points of a Test for Clusters.}}
\bjournal{{Journal of the American Statistical Association}}
\bvolume{64}
\bpages{1647-1648}.
\end{barticle}
\endbibitem

\bibitem[\protect\citeauthoryear{Garc\'{i}a-Escudero, Gordaliza and
  Matr\'{a}n}{1999}]{garcia1}
\begin{barticle}[author]
\bauthor{\bsnm{Garc\'{i}a-Escudero},~\bfnm{L~A}\binits{L.~A.}},
  \bauthor{\bsnm{Gordaliza},~\bfnm{A}\binits{A.}} \AND
  \bauthor{\bsnm{Matr\'{a}n},~\bfnm{C}\binits{C.}}
(\byear{1999}).
\btitle{{A Central Limit Theorem for Multivariate Generalized Trimmed
  $k$-means}}.
\bjournal{{Annals of Statistics}}
\bvolume{27}
\bpages{1061-1079}.
\end{barticle}
\endbibitem

\bibitem[\protect\citeauthoryear{Hartigan}{1978}]{JH2}
\begin{barticle}[author]
\bauthor{\bsnm{Hartigan},~\bfnm{J}\binits{J.}}
(\byear{1978}).
\btitle{{Asymptotic Distributions for Clustering Criteria}}.
\bjournal{{Annals of Statistics}}
\bvolume{6}
\bpages{117-131}.
\end{barticle}
\endbibitem

\bibitem[\protect\citeauthoryear{Hartigan and Hartigan}{1985}]{JH3}
\begin{barticle}[author]
\bauthor{\bsnm{Hartigan},~\bfnm{J~A}\binits{J.~A.}} \AND
  \bauthor{\bsnm{Hartigan},~\bfnm{P~M}\binits{P.~M.}}
(\byear{1985}).
\btitle{{A Dip Test of Unimodality}}.
\bjournal{{Annals of Statistics}}
\bvolume{13}
\bpages{70-84}.
\end{barticle}
\endbibitem

\bibitem[\protect\citeauthoryear{Holzmann and Vollmer}{2008}]{HV}
\begin{barticle}[author]
\bauthor{\bsnm{Holzmann},~\bfnm{H}\binits{H.}} \AND
  \bauthor{\bsnm{Vollmer},~\bfnm{Sebastian}\binits{S.}}
(\byear{2008}).
\btitle{{A Likelihood Ratio Test for Bimodality in Two-component Mixtures with
  Application to Regional Income Distribution in the EU}}.
\bjournal{{Advances in Statistical Analysis}}
\bvolume{92}
\bpages{57-69}.
\end{barticle}
\endbibitem

\bibitem[\protect\citeauthoryear{Pollard}{1981}]{DP1}
\begin{barticle}[author]
\bauthor{\bsnm{Pollard},~\bfnm{David}\binits{D.}}
(\byear{1981}).
\btitle{{Strong Consistency for $K$-Means Clustering}}.
\bjournal{{Annals of Statistics}}
\bvolume{9}
\bpages{135-140}.
\end{barticle}
\endbibitem

\bibitem[\protect\citeauthoryear{Pollard}{1982}]{DP2}
\begin{barticle}[author]
\bauthor{\bsnm{Pollard},~\bfnm{David}\binits{D.}}
(\byear{1982}).
\btitle{{A Central Limit Theorem for $k$-means Clustering}}.
\bjournal{{Annals of Statistics}}
\bvolume{10}
\bpages{919-926}.
\end{barticle}
\endbibitem

\bibitem[\protect\citeauthoryear{Pollard}{1984}]{pollard1984}
\begin{bbook}[author]
\bauthor{\bsnm{Pollard},~\bfnm{D}\binits{D.}}
(\byear{1984}).
\btitle{{Convergence of Stochastic Processes}}.
\bpublisher{Springer-Verlag, New York}.
\end{bbook}
\endbibitem

\bibitem[\protect\citeauthoryear{Schwab, Podsiadlowski and
  Rappaport}{2012}]{astro}
\begin{barticle}[author]
\bauthor{\bsnm{Schwab},~\bfnm{J}\binits{J.}},
  \bauthor{\bsnm{Podsiadlowski},~\bfnm{P~H}\binits{P.~H.}} \AND
  \bauthor{\bsnm{Rappaport},~\bfnm{S}\binits{S.}}
(\byear{2012}).
\btitle{{ Further Evidence for the Bimodal Distribution of Neutron-Star
  Masses}}.
\bjournal{{The Astrophysical Journal}}
\bvolume{719}
\bpages{722-727}.
\end{barticle}
\endbibitem

\bibitem[\protect\citeauthoryear{Serfling}{1980}]{serfling}
\begin{bbook}[author]
\bauthor{\bsnm{Serfling},~\bfnm{R}\binits{R.}}
(\byear{1980}).
\btitle{{Approximation Theorems for Mathematical Statistics}}.
\bpublisher{John Wiley, New york}.
\end{bbook}
\endbibitem

\bibitem[\protect\citeauthoryear{Serinko and Babu}{1992}]{SGJ}
\begin{barticle}[author]
\bauthor{\bsnm{Serinko},~\bfnm{R~J}\binits{R.~J.}} \AND
  \bauthor{\bsnm{Babu},~\bfnm{G~J}\binits{G.~J.}}
(\byear{1992}).
\btitle{{Weak Limit Theorems for Univariate $k$-means Clustering under
  Nonregular Conditions}}.
\bjournal{{Journal of Multivariate Analysis}}
\bvolume{49}
\bpages{188-203}.
\end{barticle}
\endbibitem

\bibitem[\protect\citeauthoryear{Stigler}{1973}]{stigler}
\begin{barticle}[author]
\bauthor{\bsnm{Stigler},~\bfnm{S~M}\binits{S.~M.}}
(\byear{1973}).
\btitle{{The Asymptotic Distribution of the Trimmed Mean}}.
\bjournal{{Annals of Statistics}}
\bvolume{1}
\bpages{472-477}.
\end{barticle}
\endbibitem

\bibitem[\protect\citeauthoryear{Wolfe}{1970}]{Wolfe}
\begin{barticle}[author]
\bauthor{\bsnm{Wolfe},~\bfnm{J~H}\binits{J.~H.}}
(\byear{1970}).
\btitle{{Pattern Clustering by Multivariate Mixture Analysis }}.
\bjournal{{Multivariate Behavioral Research}}
\bvolume{5}
\bpages{329-350}.
\end{barticle}
\endbibitem

\end{thebibliography}
\bibliographystyle{imsart-nameyear}
\end{document}